\documentclass[12pt, reqno]{amsart}
\usepackage{amsmath, amsthm, amscd, amsfonts, amssymb, graphicx, color}

\newtheorem{df}{Definition}[section]
\newtheorem{thm}[df]{Theorem}
\newtheorem{pro}[df]{Proposition}

\begin{document}
\setcounter{page}{1}

\title[Lie algebra of Operators]{On the Spectral Set of a  \\ solvable Lie algebra of Operators }
\author{Enrico Boasso}

\begin{abstract}Given $L$  a complex solvable finite dimensional Lie Algebra of operators acting on a
Banach space $E$ and $\{ x_i\}_{1\le i\le n}$  a Jordan-H\"older basis of $ L$, 
we study the relation between $Sp(L,E)$ and $\Pi Sp(x_i)$.  \end{abstract}
\maketitle
\section{Introduction}
\noindent J. L. Taylor developed in [4] a notion of joint spectrum for a $n$-tuple
$a$ = $(a_1,...,a_n)$ of mutually commuting operators acting on a Banach space 
$E$, i. e., $a_i$ $ \in$ ${\mathcal L}(E)$ ($1\le i\le n $), the algebra of all bounded linear operators acting 
on $E$, and $[a_i,a_j]=0$, $1\le i,j\le n $. This interesting notion depends
on the action of the bounded and linear maps $a_i$ on $E$ ($1\le i\le n $) and extends in a natural way the classical definition
of spectrum of a single operator. Taylor's joint spectrum, which we denote by
$Sp(a,E)$, has many remarkable properties, among then the projection property
and the fact that $Sp(a,E)$ is a compact non empty subset of $\mathbb{C}^n$. Another property, in which we are specially interested,  
 is a  well known fact about Taylor's joint spectrum, the relation between 
$Sp(a,E)$ and $Sp(a_i)$, $1\le i\le n$:
$$
Sp(a,E)\subseteq\prod_{i=1}^n Sp(a_i),
$$
where $Sp(a_i)$ denotes the spectral set of $a_i$.\par
\indent In [1] we developed a spectral theory for complex solvable
finite dimensional Lie algebras acting on a Banach space $E$. If $L$ is such
an algebra and $Sp(L,E)$ denotes its spectrum, then $Sp(L,E)$ is a compact
non empty subset of $L^*$ which also satisfies the projection property
for ideals, see [1]. Besides, when $L$ is a commutative algebra, $Sp(L,E)$ reduces to Taylor
joint spectrum in the following sense. If $\dim\, L = n$ and if $\{a_i\}_{(1\le i\le n)}$
is a basis of $L$, we consider the $n$-tuple $a = (a_1,...,a_n)$, then 
$\{(f(a_1),...,f(a_n))\colon f\in Sp(L,E)\} = Sp(a,E)$; i. e., $Sp(L,E)$, in terms
of the basis of $L^*$ dual of $\{a_i\}_{(1\le i\le n)}$, coincides with the Taylor
joint spectrum of the $n$-tuple $a$. Then, the following question arises naturally: if $\{ x_i\}_{(1\le i \le n)}$ is a basis of $L$,
and if we consider, as above, $Sp(L,E)$ in terms of the basis of $L^*$ dual of $\{x_i\}_{(1\le i\le n)}$,
i. e., if we identify $Sp(L,E)$ with its coordinate expresion $\{ (f(x_1),...,f(x_n))\colon f\in Sp(L,E)\}$, 
does $ Sp(L,E)$ satisfy the relation:

$$
\{(f(x_1),...,f(x_n))\colon f\in Sp(L,E)\}\subseteq\prod_{i=1}^n Sp(x_i)?
$$
\indent The answer, even if $\{x_i\}_{(1\le i\le n)}$ is a Jordan-H\"older basis of $L$  (Section 2),
is in general  negative.\par
\indent In this paper we study this problem, i.e., the relation between $Sp(L,E)$ and $\prod_{i=1}^n Sp(x_i)$.
 Refining an idea of [1], we describe this relation by means
of the structure of $L$ in a way that generalizes the well known result
of the commutative case. Furthermore, when $L$ is a nilpotent Lie algebra, in particular when $L$ is a commutative algebra,  we reobtain
the previous inclusion and, when $L$ is a solvable non nilpotent Lie algebra,
we give an example in order to show that our characterization can not be improved.\par
\indent The paper is organized as follows. In Section 2 we review several definitions and results 
that we need for our work.
In Section 3 we prove our main theorems for solvable  and nilpotent Lie algebras. Finally, in Section 4 we give an example in order to show that our characterization can not be improved.\par
\vskip4pt

\section{Preliminaries}

\noindent We briefly recall several definitions and results related to the 
spectrum of solvable Lie algebras of operators, see [1].  From now on, $L$ denotes a complex solvable finite dimensional  
Lie algebra and $E$  a Banach space on which $L$ acts as right continuous linear
operators, i.e., $L$ is a Lie subalgebra of ${\mathcal L}(E)$ with the opposite product. If $ \dim\, (L) = n$ and  $f$ is a character of $L$, i. e.,
$f$ belongs to $L^*$ and $f(L^2) =0$, where $L^2 = \{ [x,y]\colon x,y \in L\}$,
then we consider the following chain complex, $(E\otimes\wedge L,d(f))$, where $\wedge L$ denotes
the exterior algebra of $L$ and $d_{p-1}(f)$ is such that:

$$
d_{p-1}(f)\colon E\otimes\wedge^p L\to E\otimes\wedge^{p-1} L,
$$
 \begin{align*} 
 d_{p-1}(f)e(x_1\wedge\ldots\wedge x_p) &= \sum_{k=1}^{p}(-1)^{k+1}e(x_k-f(x_k))(x_1\wedge\ldots \wedge\Hat {x_k}\wedge\ldots \wedge x_p)\\
                                        &+\sum_{1\le k<l \le p} (-1)^{k+l} e([x_k,x_l]\wedge x_1\ldots
\wedge \Hat {x_k}\wedge \ldots\wedge \Hat {x_l}\wedge\ldots\wedge x_p) ,\\
\end{align*} 
$$
\eqno(1)
$$

\noindent where $\Hat{}$ means deletion, and $e(x_1\wedge\dots\wedge x_p)$ denotes an element of
$E\otimes\wedge^p L$. If $ p \le 0$ or $p> n$, then we also define
$d_p(f)\equiv0.$

\indent Let $H_*(E\otimes\wedge L,d(f))$ denote the homology of the complex
$(E\otimes\wedge L, d(f))$. We now state our first definition.\par
\vskip4pt
\begin{df} With $L$ and $f$ be as above, the set
$\{ f\in L^*\colon f(L^2)=0 ,H_*(E\otimes\wedge L,d(f))\ne 0\}$ is the joint spectrum
of $L$ acting on $E$ and it is denoted by $Sp(L,E)$.\end{df}
\indent As we have said, in [1] we proved that $Sp(L,E)$ is a compact non empty subset of $L^*$,
which reduces  to Taylor joint spectrum, in the sense of the Introduction, when $L$ is a
commutative algebra. Besides, if $I$ is an ideal of $L$ and $\pi$ denotes
the projection map from $L^*$ to $I^*$, then:
$$
Sp(I,E)= \pi (Sp(L,E)),
$$
i. e., the projection property for ideals still holds. With regard to this property,
I ought to mention the paper [3] of C. Ott, who pointed out a gap
in the proof of this result, and gave another proof of it. In any case, 
the projection property remains true.\par
\indent Next we recall several results
which we need for our main theorem. First, as in [1], we consider an $n-1$ dimensional ideal, $L'$, of $L$
and we decompose $E\otimes\wedge^p L $ in the following way:

$$
E\otimes\wedge^p L=\bigl(E\otimes\wedge^p L'\bigr)\oplus
\bigl( E\otimes\wedge^{p-1} L'\bigr)\wedge\langle x\rangle ,  
$$

\noindent where $ x \in L $ and is such that $L'\oplus\langle x\rangle=L$ and $\langle x\rangle$
denotes the one dimensional subspace of $L$ generated by the vector $x$.
 If $\tilde f$ denotes the restriction of $f$ to $L'$, then
we may consider the complex $\bigl(E\otimes\wedge^p L', d(\tilde f)\bigr)$ and, as $L'$ is an ideal of codimension 1 of $L$, we may decompose the
operator $d_p(f)$ as follows:

$$ 
d_{p-1}(f)\colon E\otimes\wedge^p L'\to E\otimes\wedge^{p-1} L ,'
$$

$$
 d_{p-1}(f)= d_{p-1}(\tilde f), \eqno(2)
$$

$$ 
d_{p-1}(f)\colon (E\otimes\wedge^{p-1} L')\wedge\langle x\rangle\to
E\otimes\wedge^{p-1} L'\oplus (E\otimes\wedge^{p-2} L')\wedge\langle x\rangle , 
$$

$$  
d_{p-1}(f)(a\wedge \langle x\rangle)= (-1)^{p+1} L_{p-1}(a)+
(d_{p-2}(\tilde f)(a)) \wedge \langle x\rangle, \eqno(3) 
$$

\noindent where $ a  \in E\otimes\wedge^{p-1} L'$, and $L_{p-1}$ is the
bounded linear endomorphism  defined on $  E\otimes\wedge^{p-1} L'$
by:
 \begin{align*} 
 L_{p-1} e(x_1\wedge\ldots\wedge x_{p-1})
&= e(x-f(x)) (x_1\wedge\ldots \wedge x_{p-1})\\
&+\sum_{1\le k \le p-1} (-1)^{k} e([x,x_k]\wedge x_1\ldots
\wedge \Hat {x_k}\wedge\ldots\wedge x_{p-1}) , \qquad (4)\\ 
\end{align*}

\noindent where $\Hat{}$ means deletion and $\{x_i\}_{(1\le i\le p-1)}$ belongs to $L'$.\par
\indent We use the map $\theta$ defined in [2]. We recall
the main facts which we need for our work. Let $ad(x)$, $x\in L$, be the derivation
of $L$ defined by : 
$$
ad(x)(y)= [x,y] ,\qquad\qquad (y\in L),
$$
then $\theta (x)$ is the derivation of $\wedge L$ which extends $ad(x)$, and is defined by:
$$ 
\theta(x) (x_1\wedge\ldots\wedge x_p)=
\sum_{i=1}^p (x_1\wedge\ldots\wedge ad(x)(x_i)\wedge\ldots
\wedge x_p). \eqno(5)                      
$$

\indent Observe that if we consider the map $1_E\otimes\theta(x)$,
which we still denote  by $\theta(x)$, then
\begin{align*}
L_{p-1} e(x_1\wedge\ldots\wedge x_{p-1})
&=  e(x-f(x)) (x_1\wedge\ldots\wedge x_{p-1})\\
& -\theta(x) e(x_1\wedge\ldots\wedge x_{p-1}).\qquad\qquad\qquad\qquad\qquad\qquad (6)\\ 
\end{align*}

\indent Finally, as $L$ is a complex solvable finite dimensional Lie algebra, it is well known that
there is a Jordan-H\"older sequence of ideals such that:\par
\noindent (i) $\{0\} =L_0\subseteq L_i\subseteq L_n =L$,\par
\noindent (ii) $\dim\,  L_i=i$,\par
 \noindent (iii) there is a $k$, $0\le k\le n$, such that $L_k =L^2 =[L,L]$.\par
\indent  As a consequence, if we consider a basis of $L$, $\{x_j\}_{(1\le j\le n)}$, such that
$\{x_j\}_{( 1\le j\le i)}$ is a basis of $L_i$, then we have:
$$
[x_j, x_i] =\sum_{h=1}^i c^h_{ij} x_h \qquad\qquad  (i < j).\eqno(7) 
$$
\indent Such a basis is a Jordan-H\"older basis of $L$.\par
\indent  In addition, if $L$ is a nilpotent Lie algebra, we may add the condition:
\noindent (iv) $[L,L_i]\subseteq L_{i-1}.$\par
\indent Then, in terms of the previous basis, we have:
$$
[x_j, x_i] =\sum_{h=1}^{i-1} c^h_{ij} x_h \qquad\qquad    (i < j).
\eqno(8)
$$

\section{ The Spectral Set}

\noindent First we give a definition that we need for our main theorems. We consider
 for $ p$ such that $0\le p\le n-1$, the set of p-tuples of $[\![ 1,n-1 ]\!]$, $I_p$ ,  defined as follows. If $p =0$,
$$
I_0 =\{1\},
$$ 
and for $p$ such that $1\le p \le n-1$,
$$
I_p =\{ (i_1,\ldots,i_p)\colon 1\le i_1 <\ldots < i_j < \ldots < i_p\le n-1\}.
$$
\indent We observe that $I_p$ has a natural order.\par    
\indent If $\alpha = (i_1,\ldots,i_p)$ and $\beta =(j_1,\ldots,j_p)$
belong to $I_p$, let $k = min\{ l\colon i_l\ne j_l\}$, then:\par
\noindent (i) $i_l=j_l$, $1\le l\le k-1$,\par
\noindent (ii) $i_k\ne j_k$.\par
If $i_k< j_k$ (respectively $j_k <i_k$) we put $\alpha < \beta$ (respectively 
$\beta <\alpha)$.\par
\indent Now, given $L$, $L'$, $x$, and $E$ are as in Section 2, let us consider a sequence  $\{x_i\}_{(1\le i\le n-1)}$ of elements of $L'$. 
 If $\alpha = (i_1,\ldots ,i_p)$ belongs to $I_p$, then
we  denote $ (x_{i_1}\wedge\ldots\wedge x_{i_p})$ by $(x_{\alpha})$, i.e.,
$$
(x_{\alpha})= (x_{i_1}\wedge\ldots\wedge x_{i_p}).
$$
\noindent In particular, when $p=0$, we denote $(x_0)$  by $ ( 1)$, i.e., 
$(x_0) = ( 1)$.\par

\indent In addition, as $L'$ is an ideal of $L$, $\theta (x)(\wedge L')\subseteq\wedge L'$.
Thus, we have a well defined map, which we still denote by $\theta (x)$:
$$
\theta (x)\colon E\otimes\wedge  L' \to  E\otimes\wedge L'.
$$
\indent Now, if $(L_i)_{(0\le i\le n)}$ is a Jordan-H\"older sequence of
$L$ and $\{ x_i\}_{(1\le i\le n)}$ is a Jordan-H\"older basis of $L$
associated to $(L_i)_{(0\le i\le n)}$, then we set $L'= L_{n-1}$ and $x= x_n$.\par

\indent In order to prove the following proposition we need to associate to each $\alpha$ $\in I_p$, $0\le p\le n-1$, a number $r_{\alpha}$. If $\alpha $ belongs to $I_p $, $ \alpha = (i_1,\ldots ,i_p)$, and $[x_n,x_{i_k}]=\sum_{h=1}^{i_k} c^{h}_{i_k n} x_h$,
we define for $p$ such that $1\le p\le n-1$, $r_{\alpha}= \sum^p_{k=1} c^{i_k}_{i_k n}$,
and if $p=0$, we  define $r_1=0$. Then a standard calculation shows that:
\begin{align*}
\theta (x)e(x_{\alpha}) &= X +( \sum_{k=1}^p c^{i_k}_{i_k n})e(x_{\alpha})\\
                                     &= X + r_{\alpha}e(x_{\alpha}),   
 \end{align*}
where X belongs to $\bigoplus_{\beta< \alpha} E(x_{\beta})$. \par
\indent Besides, as $x_n$ acts on $E$, $\overline x_n = x_n\otimes 1-1\otimes \theta (x_n)$
acts on $E\otimes\wedge L_{n-1}$ in a natural way, where $1$ denotes the identity map of 
the corresponding spaces. Let us compute $Sp(\overline x_n, E\otimes\wedge L_{n-1})$,
i. e., the spectrum of $\overline x_n$ in $E\otimes\wedge L_{n-1}$. If we decompose $E\otimes\wedge L_{n-1}$
by means of $E(x_{\alpha})$, $\alpha \in I_p$, $0\le p\le n-1$, then we have that $E\otimes\wedge L_{n-1} = \oplus_{(\alpha \in I_p, \hbox{ }0\le p\le n-1)} E(x_{\alpha})$.
Now, as $\theta(x_n)$ acting on $\wedge L_{n-1}$ has an upper triangular form
with diagonal entries $r_{\alpha}$, $\overline x_n$ has  in the above decomposition  an
upper triangular form with diagonal entries $x_n -r_{\alpha}$, thus, $Sp(\overline x_n,E\otimes\wedge L_{n-1})= Sp(x_n)-\{r_{\alpha}\colon \alpha\in I_p, 0\le p\le n-1\}$. 
Finally, we observe that the spectrum of $\overline x_n$ depends on the
structure of $L$ as a Lie algebra, and that in the commutative case, $\overline x_n= x_n\otimes 1$ 
and $Sp(\overline x_n,E\otimes \wedge L_{n-1}) = Sp(x_n)$.\par
\indent The first step to our main theorem is Proposition 1. 
\vskip4pt
\begin{pro} Let $ L$ be a complex  solvable finite dimensional Lie algebra, acting
 as right continuous linear operators on a
Banach space $E$. Let $(L_i)_{(0\le i \le n)}$ be a Jordan-H\"older sequence of $L$
and consider $\{x_i\}_{(1\le i \le n)}$,  a basis associated to this sequence. Then, if $f$  is a character of $L$ such that 
$$
f(x_n)\notin Sp(\overline x_n,E\otimes\wedge L_{n-1}),
$$
$f$ does not belong to $Sp(L,E)$.\end{pro}
\begin{proof}
\indent First we decompose $E\otimes\wedge^p L$ as in Section 2:
$$E\otimes\wedge^p L=\bigl(E\otimes\wedge^p L_{n-1}\bigr)\oplus
\bigl( E\otimes\wedge^{p-1} L_{n-1}\bigr)\wedge\langle x_n\rangle. $$
\indent As $L_{n-1}$ is an ideal of $L$, $ad(x_n)(L_{n-1})\subseteq L_{n-1}$ and
 $$
\theta(x_n) (E\otimes\wedge^{p-1} L_{n-1}) \subseteq E\otimes\wedge^{p-1} L_{n-1}.
$$
Then, by (4) and (5),
$$
L_{p-1} =(x_n-\theta(x_n))-f(x_n).
$$
\indent  Moreover, if we decompose $E\otimes\wedge^{p-1} L_{n-1}$ by means of $ E (x_{\alpha})$, it is obvious that:
$$
E\otimes\wedge^{p-1} L_{n-1} =\bigoplus_{\alpha \in I_{p-1}} E(\alpha).
$$
Then, by the previous considerations and the above formula, $L_{p-1}$ is an upper triangular matrix with diagonal entries $(x_n-r_{\alpha})-f(x_n)$
associated to $\alpha =(i_1,\ldots,i_{p-1}) \in I_{p-1}$.Thus, if $f$ satisfies the hypothesis,
 $L_p$ is an invertible operator for each  $p, 0\le p\le n-1.$\par
\indent We now construct a homotopy operator, $(S_p)_{ p\in Z}$, for
the complex $( E\otimes\wedge L, d(f))$, in order to show that
$H_*(E\otimes\wedge L,d(f))=0$, which is equivalent to $f\notin Sp(L,E)$.\par
\vskip4pt
\indent $S_p$ is a map from $E\otimes\wedge^p L$ to $E\otimes\wedge^{p+1} L$; we define it 
as follows:
$$
S_p\colon E\otimes\wedge^p L\to E\otimes\wedge^{p+1} L,
$$
if $p<0$ or $p>n-1$, we define $S_p\equiv 0$, if $p$ is such that $0\le p\le n-1$,
we consider the decomposition of $E\otimes\wedge^p L$ set at the beginning of the proof,
and we define
$$
S_p( E\otimes\wedge^{p-1} L_{n-1}\wedge \langle x\rangle)=0,
\eqno(9)
$$
and $S_p$ restricted to $E\otimes\wedge^p L_{n-1}$ satisfies:

$$
S_p(E\otimes\wedge^p L_{n-1})\subseteq E\otimes\wedge^p L_{n-1}\wedge\langle x_n\rangle,
$$

$$
S_p=(-1)^{p} L_p^{-1}\wedge (x_n).\eqno(10)
$$

\indent In order to verify that $S_p$ is a homotopy operator we prove the following formula:

$$
S_p d_p L_{p+1} = (-1)^{p} d_p\wedge (x_n).
\eqno(11)
$$
\indent By (2) and (3), we have
\begin{align*}
d_p L_{p+1} &= d_p (( d_{p+1} -d_p\wedge (x_n))(-1)^{p+3})\\
            &= (-1)^{p} d_p (d_p\wedge (x_n))\\
            &= (-1)^{p} (-1)^{p+2} L_p d_p\\
            &= L_p d_p.\\
\end{align*}
Then,
$$
d_p L_{p+1} =  L_p d_p.   
$$
Thus,
$$
S_p d_p L_{p+1} = S_p L_p d_p = (-1)^{p} d_p\wedge (x_n).
$$
\indent Now, by means of formulas (9), (10) and (11), an easy calculation shows that 
$$
d_pS_p +S_{p-1}d_{p-1}=I,
$$
for $p\in Z$, i. e., $(S_p)_{ p\in Z}$ is a homotopy operator .
\end{proof}

\vskip4pt
\indent In order to state our main theorem, we consider the basis $\{x_i\}_{(1\le i\le n)}$ of (7), and we
apply the definition of the beginning of the paragraph to $L_j$, the ideal generated by $\{x_i\}_{(1\le i\le j)}$,  $1\le j\le n$.
We denote by $I_p^j$, $0\le p\le j-1$, $ 1\le j\le n$, the set of p-tuples associated
to $L_j$  and the ideal $L_{j-1}$, and if $\alpha$ belongs to $I_p^j$  we denote by $r^j_{\alpha}$ the complex number associated to $\alpha$.
 In addition, we observe that in Theorem 1 and 2 below, we consider the set
$Sp(L,E)$ in terms of the basis of $L^*$ dual of $\{ x_i\}_{(1\le i\le n)}$,
i. e., we identify $Sp(L,E)$ with its coordinate expression in the mentioned basis: $\{(f(x_1),...,f(x_n))\colon f\in Sp(L,E)\}$.\par
\vskip4pt
\indent Now we state our main theorem.\par
\begin{thm} Let $L$ be a complex  solvable finite dimensional Lie
algebra, acting as right continuous linear operators on a Banach space $E$. Let $(L_i)_{(0\le
i\le n)}$ be a Jordan-H\"older sequence of $L$ and consider $\{x_i\}_{(1\le i\le n)}$,
a basis associated to this sequence. Then, in terms of the basis of $L^*$ dual of
$\{ x_i\}_{( 1\le i\le n)}$, 
$$
Sp(L,E)\subseteq \prod_{1\le j\le n} Sp(\overline x_j, E\otimes\wedge L_{j-1}).
$$
\end{thm}
\vskip4pt
 \begin{proof}
\vskip4pt
\indent By  means of an induction argument, the proof is a consequence of Proposition 1
and Theorem 3 of [1].
\end{proof}

\indent In the case of a nilpotent Lie algebra, Theorem 2 extends directly the commutative case.\par
\begin{thm} Let $L$ be a complex nilpotent finite dimensional Lie
algebra, acting as right continuous linear operators on a Banach space $E$. Let $(L_i)_{(0\le i\le n)}$
be a Jordan-H\"older sequence of $L$ and consider $\{x_i\}_{(1\le i\le n)}$, a basis
associated to this sequence. Then in terms of the basis of $L^*$
dual of $\{ x_i\}_{( 1\le i\le n)}$,
$$
Sp(L,E)\subseteq \prod_{i=1}^n Sp(x_i).
$$
 In particular,
$$
Sp(L,E)\subseteq\{ f\in L^*\colon f(L^2)=0, \hbox{ }  \parallel f(x)\parallel\leq\parallel x\parallel_{{\mathcal L}(E)}, \hbox{ }\forall x\in L\}. 
$$
\end{thm}

\begin{proof}\endproof

\indent As $L$ is a nilpotent Lie algebra, we may consider a Jordan-H\"older sequence of $L$, $ (L_j)_{ ( 0\le j\le n)} $,
 such that $[L,L_j]\subseteq L_{j-1}$. Then for each $\alpha\in I_p^j$, $1\le j\le n$, 
$0\le p\le j-1$, we have:
$$
r^j_{\alpha}=0,   
$$
which implies that $Sp(x_i) = Sp(\overline x_i, E\otimes\wedge L_{i-1})$. Thus, by means of Theorem 1 we conclude the proof.

\end{proof}
\section{ An Example}

\noindent We give an example in order to see that  Theorem 1 can not be, in general,
improved. We consider the solvable Lie algebra $G_2$ on two generators,
$$
G_2 = \langle y\rangle\oplus\langle x\rangle,
$$
with relations $[x,y]=y$. Then, by Theorem 1,
$$
Sp(G_2,E)\subseteq Sp(y)\times (Sp(x)\cup Sp(x)-1).
$$
\indent Now, if $E= \mathbb{C}^2$, and $y$ and $x$  are the following matrices:

$$
y= \begin{pmatrix}
            1&1\\
            -1&-1\\
    \end{pmatrix}, \hbox{                  }
x=\begin{pmatrix}
            0& 1/2\\
            1/2 &0 \\
   \end{pmatrix},
$$

\noindent then, $L =\langle y\rangle\oplus\langle x\rangle$ defines a Lie
subalgebra of $\mathcal L(\mathbb{C}^2)$ isomorphic to $G_2$, and an easy calculation shows that

$$
   Sp(G_2, \mathbb{C}^2) = \{0\}\times\{1/2, -3/2\}.
$$

However, as $Sp(x)= \{1/2,-1/2\}$, and $Sp(y)= 0$,  Theorem 1 cannot be,
in general, improved.

\bibliographystyle{amsplain}

\vskip.5cm

Enrico Boasso\par
E-mail address: enrico\_odisseo@yahoo.it

\end{document}